\documentclass{amsart}
\usepackage{amsmath}
\usepackage{amssymb}
\usepackage{amsthm}

\DeclareMathOperator{\im}{im}

\DeclareMathOperator{\rank}{rank}

\DeclareMathOperator{\Imaginary}{Im}
\DeclareMathOperator{\Real}{Re}

\newcommand{\oQ}{\overline{Q}}

\newcommand{\oBox}{\overline{\Box}}

\newcommand{\lp}{\langle}
\newcommand{\rp}{\rangle}
\newcommand{\lv}{\lvert}
\newcommand{\rv}{\rvert}
\newcommand{\lV}{\lVert}
\newcommand{\rV}{\rVert}

\newcommand{\mP}{\mathcal{P}}

\newcommand{\mS}{\mathcal{S}}

\newcommand{\bC}{\mathbb{C}}

\newcommand{\bN}{\mathbb{N}}

\def\sideremark#1{\ifvmode\leavevmode\fi\vadjust{\vbox to0pt{\vss
 \hbox to 0pt{\hskip\hsize\hskip1em
 \vbox{\hsize3cm\tiny\raggedright\pretolerance10000
 \noindent #1\hfill}\hss}\vbox to8pt{\vfil}\vss}}}

\newcommand{\comment}[1]{}

\newtheorem{thm}{Theorem}[section]
\newtheorem{prop}[thm]{Proposition}
\newtheorem{lem}[thm]{Lemma}
\newtheorem{cor}[thm]{Corollary}

\theoremstyle{definition}
\newtheorem{defn}[thm]{Definition}

\theoremstyle{remark}
\newtheorem{remark}[thm]{Remark}

\numberwithin{equation}{section}
\usepackage[all]{xy}

\begin{document}

\title[The CR Paneitz Operator and Stability of CR Pluriharmonic Functions]{The CR Paneitz Operator and the Stability of CR Pluriharmonic Functions}
\author{Jeffrey S.\ Case}
\address{109 McAllister Building \\ Penn State University \\ University Park, PA 16801}
\email{jscase@psu.edu}
\author{Sagun Chanillo}
\thanks{SC was partially supported by NSF Grant No.\ DMS-1201474}
\address{Department of Mathematics \\ Rutgers University \\ 110 Frelinghuysen Rd., Piscataway, NJ 08854}
\email{chanillo@math.rutgers.edu}
\author{Paul Yang}
\thanks{PY was partially supported by NSF Grant No.\ DMS-1104536}
\address{Department of Mathematics \\ Princeton University \\ Princeton, NJ 08540}
\email{yang@math.princeton.edu}
\begin{abstract}
We give a condition which ensures that the Paneitz operator of an embedded three-dimensional CR manifold is nonnegative and has kernel consisting only of the CR pluriharmonic functions.  Our condition requires uniform positivity of the Webster scalar curvature and the stability of the CR pluriharmonic functions for a real analytic deformation.  As an application, we show that the real ellipsoids in $\mathbb{C}^2$ are such that the CR Paneitz operator is nonnegative with kernel consisting only of the CR pluriharmonic functions.
\end{abstract}
\maketitle

\section{Introduction}
\label{sec:intro}

In this paper, we study two related questions about the CR Paneitz operator in dimension
three, where this operator plays an important role in the embedding question.

Throughout this paper, we use the notation and terminology in \cite{Lee1988} unless
otherwise specified.
Let $(M^3,J,\theta)$ be a closed three-dimensional pseudohermitian manifold,
where $\theta$ is a contact form and $J$ is a CR structure compatible with the
contact bundle $\xi=\ker\theta$.
The CR structure $J$ decomposes $\xi\otimes\bC$ into the $+i$- and $-i$-eigenspaces of $J$, denoted $T_{1,0}$ and $T_{0,1}$, respectively.
The Levi form $\left\langle\  ,\ \right\rangle_{L_\theta}$ is the Hermitian form on
$T_{1,0}$ defined by
$\left\langle Z,W\right\rangle_{L_\theta}=
-i\left\langle d\theta,Z\wedge\overline{W}\right\rangle$.
We can extend $\left\langle\  ,\ \right\rangle_{L_\theta}$ to $T_{0,1}$ by defining
$\left\langle\overline{Z} ,\overline{W}\right\rangle_{L_\theta}=
\overline{\left\langle Z,W\right\rangle}_{L_\theta}$ for all $Z,W\in T_{1,0}$.
The Levi form induces naturally a Hermitian form on the dual bundle of $T_{1,0}$,
denoted by $\left\langle\  ,\ \right\rangle_{L_\theta^*}$,
and hence on all the induced tensor bundles.  By integrating the Hermitian form
(when acting on sections) over $M$ with respect to the volume form $dV=\theta\wedge d\theta$,
we get an inner product on the space of sections of each tensor bundle.
We denote this inner product by $\left\langle \ ,\ \right\rangle$. For example
\begin{equation}\label{21}
\left\langle\varphi ,\psi\right\rangle=\int_{M}\varphi\bar{\psi}\ dV,
\end{equation}
for functions $\varphi$ and $\psi$.

The Reeb vector field $T$ is the unique vector field such that $\theta(T)=1$ and $d\theta(T,\cdot)=0$.  Let $Z_1$ be a local frame of $T_{1,0}$ and consider the frame $\left\{T,Z_1,Z_{\bar 1}\right\}$ of $TM\otimes\bC$. Then $\left\{\theta,\theta^1,\theta^{\bar 1}\right\}$,
the coframe dual to $\left\{T,Z_1,Z_{\bar{1}}\right\}$, satisfies
\begin{equation}\label{22}
d\theta=ih_{1\bar 1}\theta^1\wedge\theta^{\bar 1}
\end{equation}
for some positive function $h_{1\bar 1}$. We can always choose $Z_1$
such that $h_{1\bar 1}=1$; hence, throughout this paper, we assume
$h_{1\bar 1}=1$

The pseudohermitian connection of $(J,\theta)$ is 
$\nabla$ on $TM\otimes\bC$ (and extended to tensors), and is given in terms of a local
frame $Z_1\in T_{1,0}$ by
\begin{equation*}
\nabla Z_1=\theta_1{}^1\otimes Z_1,\quad
\nabla Z_{\bar{1}}=\theta_{\bar{1}}{}^{\bar{1}}\otimes Z_{\bar{1}},\quad
\nabla T=0,
\end{equation*}
where $\theta_1{}^1$ is the $1$-form uniquely determined by the equations
\begin{align*}
d\theta^1&=\theta^1\wedge\theta_1{}^1+\theta\wedge\tau^1 , \\
\tau^1&\equiv 0\mod\theta^{\bar 1} , \\
0&=\theta_1{}^1+\theta_{\bar{1}}{}^{\bar 1} .
\end{align*}
$\theta_{1}{}^{1}$ and $\tau^1$ are called the connection form and the pseudohermitian
torsion, respectively.
Put $\tau^1=A^1{}_{\bar 1}\theta^{\bar 1}$.
The structure equation for the pseudohermitian connection is
\begin{equation*}
d\theta_1{}^1=R\theta^1\wedge\theta^{\bar 1}
+2i\Imaginary (A^{\bar{1}}{}_{1,\bar{1}}\theta^1\wedge\theta),
\end{equation*}
where $R$ is the (Tanaka-)Webster curvature.

We denote components of covariant derivatives with indices preceded by a comma;
thus we write $A^{\bar{1}}{}_{1,\bar{1}}\theta^1\wedge\theta$.
The indices $\{0, 1, \bar{1}\}$ indicate derivatives with respect to $\{T, Z_1, Z_{\bar{1}}\}$.
For derivatives of a scalar function, we omit the comma;
for example, given a smooth function $\varphi$, we write $\varphi_{1}=Z_1\varphi$ and $\varphi_{1\bar{1}}=
Z_{\bar{1}}Z_1\varphi-\theta_1^1(Z_{\bar{1}})Z_1\varphi$ and $\varphi_{0}=T\varphi$.

We recall several natural differential operators occurring in this paper.  For a smooth function $\varphi$, the Cauchy-Riemann operator $\partial_{b}$ can be
defined locally by
\[\partial_{b}\varphi=\varphi_{1}\theta^{1}.\]
We write $\bar{\partial}_{b}$ for the conjugate of
$\partial_{b}$. A function $\varphi$ is called CR holomorphic if
$\bar{\partial}_{b}\varphi=0$. The divergence operator $\delta_b$
takes $(1,0)$-forms to functions by
$\delta_b(\sigma_{1}\theta^{1})=\sigma_{1,}{}^{1}$; similarly,
$\bar{\delta}_b(\sigma_{\bar 1}\theta^{\bar 1})=\sigma_{\bar
1,}{}^{\bar 1}$.

The Kohn Laplacian on functions is
\[\Box_{b}=2\bar{\partial}_{b}^{*}\bar{\partial_{b}} . \]
The sublaplacian is $\Delta_b=\Real\Box_b$ and the CR conformal Laplacian is $L=-\Delta_b+R/4$.

Define
\begin{equation}
\label{eqn:P3}
P_3\varphi = \left(\varphi_{\bar1}{}^{\bar1}{}_1 + iA_{11}\varphi^1\right)\theta^1 .
\end{equation}
The importance of this operator is that the space $\mP$ of CR pluriharmonic functions can be characterized as $\mP=\ker P_3$; see~\cite{Lee1988}.

The CR Paneitz operator $P_4$, first introduced by Graham and Lee~\cite{GrahamLee1988}, is
\begin{equation}
\label{eqn:P4_defn}
P_4\varphi = \delta_b\left(P_3\varphi\right) .
\end{equation}
Define $Q$ by $Q\varphi = 2i\left(A^{11}\varphi_1\right)_{,1}$.  Using the commutation relation $[\Box_b,\overline{\Box}_b]=4i\Imaginary Q$, we see that
\begin{align*}
P_4\varphi&=\frac{1}{4}(\Box_{b}\overline{\Box}_{b}-2Q)\varphi\\
&=\frac{1}{8}\big((\overline{\Box}_{b}\Box_{b}+\Box_{b}\overline{\Box}_{b})\varphi+8\Imaginary (A^{11}\varphi_{1})_{1}\big).
\end{align*}
Hence $P_4$ is a real and symmetric operator. It plays an important role in the
embedding problem: whether the CR structure can be embedded into $\bC^n$ for some integer $n$. Note that the leading order term of $P_4$ makes it a fourth order hyperbolic operator, thus it is remarkable that it still displays properties of a subelliptic operator.  

\begin{defn}
A pseudohermitian manifold $(M^3,J,\theta)$ has nonnegative CR Paneitz operator, written $P_4\geq0$, if
\[ \left\lp P_4\varphi, \varphi\right\rp \geq 0 \]
for all smooth functions $\varphi$.
\end{defn}

 In previous work~\cite{ChanilloChiuYang2010} on the embedding problem, 
Chiu and the second two authors showed that three-dimensional CR manifolds with positive CR Yamabe constant and nonnegative CR Paneitz operator are embeddable in $\bC^n$.  As a partial converse, one would like to know if CR manifolds embedded in $\bC^2$ with some additional nice properties satisfy these nonnegativity conditions.  Working in this direction, Chiu and the second two authors showed~\cite{ChanilloChiuYang2013} that these nonnegative conditions hold for small deformations of a strictly pseudoconvex hypersurface with vanishing torsion in $\bC^2$.

Another question concerning the CR Paneitz operator is the identification of its kernel.  It follows
from its definition that, on a three-dimensional CR manifold, the space of CR pluriharmonic functions is contained in the kernel of the CR Paneitz operator.  Moreover, Graham and Lee showed~\cite{GrahamLee1988} that if a three-dimensional CR manifold admits a torsion-free contact form, then the kernel of the CR Paneitz operator consists solely of the CR pluriharmonic functions.  One would like to characterize CR manifolds for which this equality holds.  Since there are known non-embedded examples for which the equality does not hold~\cite{ChanilloChiuYang2010}, we restrict our attention to embedded CR manifolds.  Motivated by this problem, Hsiao showed~\cite{Hsiao2014} that for embedded CR manifolds, there is a finite-dimensional vector space $W$ such that the kernel of the CR Paneitz operator $P_4$ splits in the Folland--Stein space $S^{2,2}$ as a direct sum,
\begin{equation}
\label{eqn:supplementary}
\ker P_4 = \mP \oplus W .
\end{equation}
For an elementary proof, see Lemma~\ref{lem:2.8}.  Generic results about the existence of $W$ may be found in~\cite{CaseChanilloYang2015b}.

We need one more definition before we can state our main result:

\begin{defn}
The space of CR pluriharmonic functions is stable for the one-parameter family $(M^3,J^t,\theta)$ of pseudohermitian manifold if for every $\varphi\in\mP^t$ and every $\varepsilon>0$, there is a $\delta>0$ such that for each $s$ satisfying $\lv t-s\rv < \delta$, there is a CR pluriharmonic function $f_s\in\mP^s$ such that
\[ \left\lV \varphi - f_s\right\rV_2 < \varepsilon . \]
\end{defn}

\begin{thm}
\label{thm:main_thm}
Let $(M^3,J^t,\theta)$ be a family of embedded CR manifolds for $t\in[-1,1]$ with the following properties.
\begin{enumerate}
\item $J^t$ is real analytic in the deformation parameter $t$.
\item The Szeg{\H o} projectors $S^t\colon F^{2,0}\to(\ker \bar\partial_b^t\subset F^{2,0})$ vary continuously in the deformation parameter $t$ (see~Section~\ref{sec:estimates} for a definition of $F^{2,0}$).
\item For the structure $J^0$ we have $P_4^0\geq0$ and $\ker P_4^0=\mP^0$, the space of CR pluriharmonic functions with respect to $J^0$.
\item There is a uniform constant $c>0$ such that
\begin{equation}
\label{eqn:uniform_curv}
\inf_{t\in[-1,1]} \min_M R^t \geq c > 0 .
\end{equation}
\item The CR pluriharmonic functions are stable for the family $(M^3,J^t,\theta)$.
\end{enumerate}
Then $P_4^t\geq0$ and $\ker P_4^t=\mP^t$ for all $t\in[-1,1]$.
\end{thm}

\begin{remark}
The assumption~\eqref{eqn:uniform_curv} can be replaced by the assumption that the CR Yamabe constants $Y[J^t]$ are uniformly positive.  Since the assumptions on the CR Paneitz operator are CR invariant~\cite{Hirachi1990}, this allows us to recast Theorem~\ref{thm:main_thm} in a CR invariant way.
\end{remark}

As an application, consider the family of real ellipsoids in $\bC^2$ as deformations of the standard CR three-sphere.  The formula established in~\cite[Theorem~1]{KerzmanStein1978} expressing the Szeg{\H o} kernel in terms of the defining function implies that condition (2) holds.  Since the standard contact form on the CR three-sphere is torsion-free, its CR Paneitz operator is nonnegative and has kernel consisting only of the CR pluriharmonic functions~\cite{GrahamLee1988}. An elementary calculation~\cite{ChanilloChiuYang2013} shows that the real ellipsoids have positive Webster scalar curvature, and thus satisfy condition (4) of Theorem~\ref{thm:main_thm}. The condition (5) then follows from the stability results of~\cite{ChanilloChiuYang2010,Lempert1994}.

\begin{cor}
\label{cor:convex}
The real ellipsoids in $\bC^2$ are such that the Paneitz operator is 
nonnegative and has kernel 
consisting only of the CR pluriharmonic functions.
\end{cor}


\section{The proof of Theorem~\ref{thm:main_thm}}
\label{sec:estimates}

The proof of Theorem~\ref{thm:main_thm} is based on a continuity argument.  Let
\begin{equation}
\label{eqn:S}
\mS = \left\{ t\in[-1,1] \colon \text{$P_4^t\geq0$ and $\ker P_4^t=\mP^t$} \right\} .
\end{equation}
By hypothesis, $0\in\mS$.  Our goal is to show that $\mS$ is open and closed, whence $S=[-1,1]$.  To do so requires a number of new estimates for deformations of CR structures.

Many of our estimates do not require the assumptions of Theorem~\ref{thm:main_thm}.  In particular, the assumption that $J^t$ is real analytic in $t$ is only used to prove that $\mS$ is open, while the continuity of the Szeg{\H o} projectors is only used to prove that $\mS$ is closed.  With the expectation that our estimates will be useful in other contexts, we isolate them below with the minimal required hypotheses.

We begin by developing the framework from which we will show that the set $\mS$ is open.  The first ingredient we need is the following subelliptic estimate for $P_4$ for functions in $\mP^\perp$ which is uniform in
\[ \lambda_1\left(\Box_b\right) := \inf\left\{ \int\lv\bar\partial_b u\rv^2 \colon \int u^2=1 \text{ and } u\in S^{1,2}\cap\left(\ker\bar\partial_b\right)^\perp \right\} . \]
Note that $\lambda_1\left(\Box_b\right)$ need not be positive for an arbitrary closed pseudohermitian manifold, though it is positive when the manifold is also embeddable~\cite{Kohn1986}.  Our estimate improves a result of Saotome and Chang~\cite{ChangSaotome2011}.

\begin{lem}
\label{lem:2.1}
Assume that $(M^3,J^t,\theta)$ is a family of pseudohermitian manifolds such that for all $t$ it holds that
\begin{equation}
\label{eqn:lem:2.1_assumption}
\lambda_1\left(\Box_b^t\right)\geq c > 0 .
\end{equation}
Then for $f\in S_t^{4,2}\cap\mP^\perp$, there exists a constant $c_1>0$, independent of $t$, such that
\begin{equation}
\label{eqn:lem:2.1}
c_1 \lV f\rV_{S^{4,2}} \leq \lV P_4^t f\rV_2 + \lV f\rV_2
\end{equation}
for $S^{4,2}$ the Folland--Stein space defined in~\cite{FollandStein1974}.
\end{lem}

\begin{proof}

Since $f\perp\mP^t$, it also holds that $f\perp\ker\partial_b^t$; i.e.\ $f$ is perpendicular to the anti-CR functions.  We can then find a solution $\psi$ to the equation $\overline{\Box}_b^t\psi=f$.  Let $h$ be an anti-CR function.  Then
\begin{align*}
\left\lp \Box_b^t f, h\right\rp & = \left\lp \Box_b^t \overline{\Box}_b^t\psi, h\right\rp \\
& = \left\lp \psi, \left(\oBox_b^t\Box_b^t - 2\oQ^t\right)h \right\rp + 2\left\lp \psi, \oQ^th \right\rp
\end{align*}
Recall that $4P_4^t = \overline{\Box}_b^t\Box_b^t - 2\oQ^t$.  Since $h$ is an anti-CR function, $h\in\ker P_4^t$.  Thus
\[ \left\lp \Box_b^t f - 2Q^t\psi, h \right\rp = 0 \]
for every $h\in\ker\partial_b^t$.

Next, since $f\perp\mP^t$, it also holds that $f\perp\ker\bar\partial_b^t$.  From the assumption~\eqref{eqn:lem:2.1_assumption}, it follows that
\[ c_1\lV f\rV_{S^{4,2}} \leq \lV\Box_b^tf\rV_{S^{2,2}} + \lV f\rV_2 \]
for a constant $c_1>0$ independent of $t$ (cf.\ \cite{ChenShaw2001}).  Therefore
\begin{equation}
\label{eqn:2.2}
c_1\lV f\rV_{S^{4,2}(M)} \leq \lV\Box_b^tf - 2Q^t\psi\rV_{S^{2,2}} + \lV Q^t\psi\rV_{S^{2,2}} + \lV f\rV_2 .
\end{equation}
Since $\Box_b^tf - 2Q^t\psi$ is orthogonal to the anti-CR functions and since the assumption~\eqref{eqn:lem:2.1_assumption} yields the same uniform lower bound on $\lambda_1\bigl(\oBox_b^t\bigr)$, it follows that
\begin{equation}
\label{eqn:2.3}
c_2\left\lV \Box_b^tf - 2Q^t\psi\right\rV_{S^{2,2}} \leq \left\lV\oBox_b^t\left(\Box_b^tf-2Q^t\psi\right)\right\rV_2 + \left\lV\Box_b^tf - 2Q^t\psi\right\rV_2
\end{equation}
for a constant $c_2>0$ independent of $t$.  Combining~\eqref{eqn:2.2} and~\eqref{eqn:2.3} we find a constant $c_3>0$, independent of $t$, such that
\begin{equation}
\label{eqn:2.4}
c_3\left\lV f\right\rV_{S^{4,2}} \leq \left\lV P_4^tf\right\rV_2 + \left\lV \oQ^t f\right\rV_2 + \left\lV Q^t\psi\right\rV_{S^{2,2}} + \left\lV f\right\rV_{S^{2,2}} + \left\lV f\right\rV_2 .
\end{equation}
Since $Q^t$ is a second-order operator, we have that
\[ \left\lV Q^t\phi \right\rV_2 \leq \left\lV \phi\right\rV_{S^{2,2}} + \left\lV\phi\right\rV_2 \]
for any function $\phi$.  On the other hand, the definition of $\psi$ gives
\[ c\left\lV\psi\right\rV_{S^{4,2}} \leq \left\lV f\right\rV_{S^{2,2}} + \left\lV f\right\rV_2 \]
for a constant $c>0$ independent of $t$.  Inserting these two estimates into~\eqref{eqn:2.4} yields a constant $c_4>0$, independent of $t$, such that
\begin{equation}
\label{eqn:2.6}
c_4\left\lV f\right\rV_{S^{4,2}} \leq \left\lV P_4^tf\right\rV_2 + \left\lV f\right\rV_{S^{2,2}} + \left\lV f\right\rV_2 .
\end{equation}

Next, an interpolation inequality in~\cite{Lu1997,Lu2000} shows that for any $\varepsilon>0$, there is a constant $C=C(\varepsilon)>0$ such that
\begin{equation}
\label{eqn:2.7}
\left\lV f\right\rV_{S^{2,2}} \leq \varepsilon\left\lV f\right\rV_{S^{4,2}} + C\left\lV f\right\rV_2 .
\end{equation}
Combining~\eqref{eqn:2.6} and~\eqref{eqn:2.7} with $\varepsilon$ sufficiently small yields the conclusion.
\end{proof}

From Lemma~\ref{lem:2.1} we recover the decomposition~\eqref{eqn:supplementary} of the kernel of the CR Paneitz operator.

\begin{lem}
\label{lem:2.8}
Let $(M^3,J,\theta)$ be an embedded CR structure.  Then
\[ \ker P_4 = \mP \oplus W \]
in $S^{2,2}$, where $\dim W<\infty$.
\end{lem}

\begin{proof}

Consider the constant family $(M^3,J^t,\theta)$ of pseudohermitian manifolds with $J^t=J$ for all $t$.  Since $(M^3,J,\theta)$ is embedded, the assumption~\eqref{eqn:lem:2.1_assumption} holds~\cite{Kohn1986}.  Let $W$ consist of all elements $f\in\ker P_4$ such that $f\perp\mP$.  From Lemma~\ref{lem:2.1} and a simple density argument we see that, for any $f\in W$,
\[ \lV f\rV_{S^{4,2}} \leq c\lV f\rV_2 . \]
Thus the unit $L^2$-ball in $W$ satisfies $\lV f\rV_{S^{4,2}}\leq c$.  Hence, by the Rellich lemma, the unit $L^2$-ball in $W$ is compact, whence $W$ is finite-dimensional.
\end{proof}

\begin{defn}
Let $(M^3,J,\theta)$ be an embedded CR structure.  The \emph{supplementary space} $W$ is the subspace from Lemma~\ref{lem:2.8}.
\end{defn}

We now show that $\mS$ is open.

\begin{prop}
\label{prop:2.9}
Let $(M^3,J^t,\theta)$ be a family of embedded CR manifolds satisfying hypotheses (1), (3), (4), and (5) of Theorem~\ref{thm:main_thm} and define $\mS$ by~\eqref{eqn:S}.  Then $\mS$ is open.
\end{prop}

\begin{proof}

Let $t_0\in\mS$.  Thus $P_4^{t_0}\geq0$ and $\ker P_4^{t_0}=\mP^{t_0}$.  Since the CR pluriharmonic functions are stable and the CR structures $J^t$ are real analytic in $t$, we may apply~\cite[Theorem~1.7]{ChanilloChiuYang2013} to conclude that there is a constant $\delta>0$ such that for any $t$ with $\lv t-t_0\rv<\delta$, it holds that
\begin{equation}
\label{eqn:2.10}
P_4^t \geq 0 .
\end{equation}
Note that the assumption in \cite[Theorem~1.7]{ChanilloChiuYang2013} that the deformation consists of manifolds embedded in $\bC^2$ was only used to invoke Lempert's stability theorem~\cite{Lempert1994} for CR functions.  Thus the conclusion of \cite[Theorem~1.7]{ChanilloChiuYang2013} holds by replacing the assumption that the manifolds are embedded in $\bC^2$ by the assumption, as in Theorem~\ref{thm:main_thm}, that the CR pluriharmonic functions are stable.

To complete the proof, we must show that for $\lv t-t_0\rv<\delta$, it holds that $\ker P_4^t = \mP^t$.  Since our structures are embedded, Lemma~\ref{lem:2.8} implies that the supplementary space $W^t$ is finite-dimensional.  By~\eqref{eqn:2.10} and the assumption that $R^t$ is uniformly bounded below by a positive constant, the main result of~\cite{ChanilloChiuYang2010} implies that there is a constant $c>0$, independent of $t$, such that
\[ \lambda_1\left(\Box_b^t\right)\geq c . \]
Thus the assumptions of Lemma~\ref{lem:2.1} hold.

Suppose to the contrary that it is not true that $W^t=\{0\}$ for all $\lv t-t_0\rv<\delta$.  Let $f_{t_k}\in W^{t_k}$ be such that $\lV f_{t_k}\rV_2=1$ and $t_k\to t_0$.  Since $f_{t_k}\perp\mP^{t_k}$, Lemma~\ref{lem:2.1} implies that $\lV f_{t_k}\rV_{S^{4,2}} \leq c$.  Thus, by the Rellich lemma, there is a function $f_0\in S^{4,2}$ such that $\lV f_0\rV_2=1$ and $f_{t_k}\to f_0$ strongly in $L^2$.  Since $P_4^{t_k} f_{t_k}=0$, we have that $P_4^{t_0}f_0=0$.  Hence, by the assumption $t_0\in\mS$, we have that $f_0\in\mP^{t_0}$.  By the stability assumption, give $\varepsilon>0$, there is a constant $\delta_1>0$ such that for all $t$ satisfying $\lv t-t_0\rv<\delta_1$, there is a function $\psi_t\in\mP^t$ such that
\begin{equation}
\label{eqn:2.11}
\lV f_0 - \psi_t\rV_2 < \varepsilon .
\end{equation}
Hence
\begin{equation}
\label{eqn:2.12}
1 = \lV f_0\rV_2^2 = \lp f_0 - \psi_{t_k}, f_0\rp + \lp \psi_{t_k}, f_0 - f_{t_k}\rp + \lp \psi_{t_k}, f_{t_k} \rp .
\end{equation}
That $f_{t_k}\in W^{t_k}$ and $\psi_{t_k}\in\mP^{t_k}$ implies that $\lp\psi_{t_k},f_{t_k}\rp=0$.  On the other hand, \eqref{eqn:2.11} implies that $\lV\psi_{t_k}\rV_2\leq C$.  Inserting these into~\eqref{eqn:2.12} and recalling that $\lV f_0\rV_2=1$ yields
\[ 1 \leq \lV f_0 - \psi_{t_k}\rV_2 + C\lV f_{t_k} - f_0\rV_2 . \]
Letting $t_k\to t_0$ thus yields the desired contradiction.
\end{proof}

\begin{remark}
The assumption that $J^t$ varies real-analytically in $t$ is only utilized to show that $\mS$ is open.  It enters only because we use \cite[Theorem~1.7]{ChanilloChiuYang2013}.
\end{remark}

We now need to establish that $\mS$ is closed.  We begin with a preliminary lemma.

\begin{lem}
\label{lem:2.15}
Let $(M^3,J^t,\theta)$ be a family of CR structures with uniformly positive Webster scalar curvature.  Let $\{t_n\}_{n=1}^\infty$ be such that $P_4^{t_n}\geq0$ and suppose that $t_n\to t_0$.  Then
\begin{enumerate}
\item $P_4^{t_0} \geq 0$.
\item $\displaystyle\limsup_{t_n\to t_0} \dim W^{t_n} \leq \dim W^{t_0}$.
\end{enumerate}
\end{lem}

\begin{proof}

First, since $t_n\in\mS$, it holds that $\lp P_4^{t_n}\varphi,\varphi\rp\geq0$ for any smooth real-valued function $\varphi$.  The Dominated Convergence Theorem yields the first claim.

Next, as in the proof of Proposition~\ref{prop:2.9}, the main result of~\cite{ChanilloChiuYang2010} yields a constant $c>0$, independent of $t_n$ and $t_0$, such that
\[ \lambda_1\left(\Box_b^{t_n}\right)\geq c . \]
Thus, by Lemma~\ref{lem:2.1}, there is a constant $c>0$, independent of $t_n$, such that
\begin{equation}
\label{eqn:2.16}
c\lV f\rV_{S^{4,2}} \leq \left\lV P_4^{t_n}f\right\rV_2 + \lV f\rV_2
\end{equation}
for all $f\perp\mP^{t_n}$.  Set
\[ \limsup_{t_n\to t_0} \dim W^{t_n} = N , \]
where $N\in\bN\cup\{0,\infty\}$.  Let $s\in\bN$ be such that $s<N$.  By taking a subsequence if necessary, we may suppose that $\dim W^{t_n}\geq s$ for all $n$.  Let $\{f_j^{t_n}\}_{j=1}^s$ be an orthonormal (with respect to the $L^2$-metric) set of functions in $W^{t_n}$.  From~\eqref{eqn:2.16} we conclude that $\lV f_j^{t_n}\lV_{S^{4,2}}\leq c$ for some constant $c$, independent of $t_n$.  Hence the Rellich lemma implies that there are functions $f_j^0\in S^{4,2}$ such that
\[ \lV f_j^{t_n} - f_j^0\rV_2 \to 0 \]
as $t_n\to t_0$ for all $j\in\{1,2,\dotsc,s\}$.  Moreover, that $\{f_j^{t_n}\}_{j=1}^s$ is orthonormal in $W^{t_n}$ implies that $\{f_j^0\}$ are orthonormal in $\ker P^{t_0}$.  Indeed, by using the stability assumption as in the proof of Proposition~\ref{prop:2.9}, we see that $\{f_j^0\}_{j=1}^s$ are orthonormal in $W^{t_0}$.  Hence
\begin{equation}
\label{eqn:2.20}
s \leq \dim W^{t_0} .
\end{equation}
Since $s<N$ is arbitrary, this yields the desired result.
\end{proof}

Following Rumin~\cite{Rumin1994} and Garfield and Lee~\cite{GarfieldLee1998}, we now consider the (bigraded) Rumin complex.  To that end, let $E^{j,k}$ for $0\leq j+k\leq 1$ and $F^{\ell,m}$ for $2\leq \ell+m\leq 3$ be the vector bundles
\begin{align*}
E^{0,0} & = \lp 1 \rp , & E^{1,0} & = \lp \theta^1\rp , \\
E^{0,1} & = \lp \theta^{\bar 1} \rp, & F^{2,0} & = \lp \theta^1\wedge\theta \rp, \\
F^{1,1} & = \lp \theta^{\bar 1}\wedge\theta \rp, & F^{2,1} & = \lp \theta\wedge d\theta \rp ,
\end{align*}
where $\lp\theta^1\rp$ denotes the span of $\theta^1$ as a $C^\infty(M;\bC)$-module.  The Rumin complex is the bigraded complex
\begin{equation}
\label{eqn:garfield_lee}
\xymatrix{
& E^{1,0} \ar[r]^{D^{\prime}} \ar[rdd]_<<<<{D^{\prime\prime}} & F^{2,0} \ar[rd]^{d^{\prime\prime}} & \\
E^{0,0} \ar[ru]^{d^\prime} \ar[rd]_{d^{\prime\prime}} & & & F^{2,1} \\
& E^{0,1} \ar[ruu]^<<<<<<{D^+} \ar[r]_{D^\prime} & F^{1,1} \ar[ru]_{d^{\prime}} & .
}
\end{equation}
where $d^\prime=\partial_b$, $d^{\prime\prime}=\bar\partial_b$, and the operators $D^\prime$, $D^{\prime\prime}$, and $D^+$ are the second order operators
\begin{subequations}
\label{eqn:Rumin}
\begin{align}
\label{eqn:Rumin:10-20} D^\prime(\sigma_1\theta^1) & = \left(-i\sigma_{1,\bar 11} - \sigma_{1,0}\right)\theta^1\wedge\theta , \\
\label{eqn:Rumin:10-11} D^{\prime\prime}(\sigma_1\theta^1) & = \left(-i\sigma_{1,\bar1\bar1} - A_{\bar1\bar1}\sigma_1\right)\theta^{\bar 1}\wedge\theta , \\
\label{eqn:Rumin:01-11} D^\prime(\sigma_{\bar 1}\theta^{\bar 1}) & = \left(i\sigma_{\bar 1,1\bar1} - \sigma_{\bar 1,0}\right)\theta^{\bar 1}\wedge\theta , \\
\label{eqn:Rumin:01-20} D^+(\sigma_{\bar 1}\theta^{\bar 1}) & = \left(i\sigma_{\bar 1,11} - A_{11}\sigma_{\bar 1}\right)\theta^1\wedge\theta .
\end{align}
\end{subequations}
\eqref{eqn:garfield_lee} is a bigraded complex in the sense that sums of compositions with the same domain and codomain vanish.  For example, $D^\prime d^\prime + D^+ d^{\prime\prime} = 0$.

We will use the Rumin complex to obtain a new characterization of the CR pluriharmonic functions.  First recall that $\mP = \ker P_3$ for $P_3$ as in~\eqref{eqn:P3}.  We then rewrite $P_3$ in terms of the operators appearing in~\eqref{eqn:garfield_lee}.

\begin{lem}
\label{lem:2.22}
Let $(M^3,J,\theta)$ be a pseudohermitian manifold.  Identify $E^{1,0}\cong F^{2,0}$ via $\theta^1\cong\theta^1\wedge\theta$ and identify $E^{0,0}\cong F^{2,1}$ via $1\cong\theta\wedge d\theta$.  Then
\begin{align*}
P_3 & = -iD^+d^{\prime\prime} , \\
P_4 & = -id^{\prime\prime}D^+d^{\prime\prime} .
\end{align*}
\end{lem}

\begin{proof}

From~\eqref{eqn:Rumin:01-20} we compute that
\[ -iD^+d^{\prime\prime} = f_{\bar111} + iA_{11}f_{\bar1}, \]
establishing the first claim.  This and the definition $P_4=\delta_bP_3$ yields the second claim.
\end{proof}

Our goal is to give an alternative description of the finite-dimensional supplementary $W^t$ in terms of the range of $P_3^t$ in~\eqref{eqn:garfield_lee}.

\begin{lem}
\label{lem:2.23}
Let $(M^3,J,\theta)$ be a pseudohermitian manifold and let $W$ be the supplementary space.  Then
\[ W \cong \ker d^{\prime\prime} \cap \im D^+d^{\prime\prime} \subset F^{2,0} . \]
Thus
\[ W \cong \ker d^{\prime\prime} \cap \im P_3 \subset F^{2,0} . \]
\end{lem}

\begin{proof}

Let $\Phi\colon W\to\ker d^{\prime\prime}\cap\im P_3\subset F^{2,0}$ be the map $\Phi(f)=P_3(f)$.  Note that $d^{\prime\prime}P_3f=P_4f=0$ since $f\in W$, and so $\Phi$ is well-defined.

To see that $\Phi$ is surjective, let $u\in\ker d^{\prime\prime}\cap\im P_3$.  We may thus find a function $h\in E^{0,0}$ such that $u=P_3h$.  By the assumption on $u$, we have that $h\in\ker P_4$, and hence we may write $h=\phi+\psi$ for $\phi\in\mP$ and $\psi\in W$.  Since $\phi\in\ker P_3$, we have that
\[ u = P_3h = P_3\psi . \]
Hence $\Phi$ is surjective.

To see that $\Phi$ is injective, assume $P_3w_1=P_3w_2$ for $w_1,w_2\in W$.  Then $P_3(w_1-w_2)=0$.  Thus $w_1-w_2\in\mP\cap W$, whence $w_1=w_2$.
\end{proof}

Now set
\[ F^t := \ker \bar\partial_b^t \cap \im P_3^t \subset F^{2,0} . \]
By Lemma~\ref{lem:2.23} we have that
\begin{equation}
\label{eqn:2.24}
\dim F^t = \dim W^t .
\end{equation}
We then have the following lemma.

\begin{lem}
\label{lem:2.26}
Let $(M^3,J^t,\theta)$ be a family of embedded CR structures such that $J^t$ is $C^6$ in the deformation parameter $t$ and the Szeg{\H o} projector $S^t\colon F^{2,0}\to\ker d^{\prime\prime}$ is continuous in the deformation parameter $t$.  Then $\dim F^t$ is a lower semi-continuous function of the deformation variable $t$.
\end{lem}

\begin{proof}

Let $S^t$ denote the Szeg{\H o} projector in $F^{2,0}$.  Consider the linear operator
\[ A^t := S^t\circ P_3^t . \]
As a consequence of Lemma~\ref{lem:2.8} and~\eqref{eqn:2.24}, we have that $\rank A^t = \dim F^t < \infty$, and hence the range of $A^t$ is finite-dimensional.  By hypothesis, $S^t$ and $P_3^t$ are continuous in $t$.  We conclude that for fixed sections $\phi\in C^3(M;\bC)$ and $\psi\in L^2\left(M;(F^{2,0})^\ast\right)$, the function
\[ h(t) := \left\lp A^t(\varphi),\psi \right\rp \]
is continuous.  We now show that $\rank A^t$ is lower semi-continuous.  To prove this fact, it suffices to show that
\[ G = \left\{ t\in[-1,1] \colon \rank A^t > a \right\} \]
is an open set for any $a$.  Fix $a$ and let $t_0\in G$.  Set $s=\rank A^{t_0}$, so that $a<s<\infty$.  Thus we can find functions $\{\phi_i\}_{i=1}^s$ and functionals $\{\psi_i\}_{i=1}^s$ so that the $s\times s$-matrix $\left(h_{ij}(t_0)\right)$ with entries
\[ h_{ij}(t_0) = \left\lp A^{t_0}\phi_i,\psi_j\right\rp \]
satisfies
\[ \det \left(h_{ij}(t_0)\right) \not= 0 . \]
Since $h(t)$ is continuous, there is a constant $\delta>0$ such that for any $t$ with $\lv t-t_0\rv<\delta$ it holds that
\[ \det\left(h_{ij}(t)\right) \not= 0 . \]
Hence $\rank A^t\geq s>a$; i.e.\ $G$ is open.
\end{proof}

\begin{cor}
\label{cor:2.27}
Let $(M^3,J^t,\theta)$ be as in Lemma~\ref{lem:2.26}.  Then $\dim W^t$ is a lower semi-continuous function of $t$.
\end{cor}

\begin{proof}

By~\eqref{eqn:2.24} we have that $\dim W^t=\dim F^t$.  The result then follows from Lemma~\ref{lem:2.26}.
\end{proof}

\begin{prop}
\label{prop:2.29}
Let $(M^3,J^t,\theta)$ be a family of embedded CR manifolds satisfying hypotheses (2) and (3) of Theorem~\ref{thm:main_thm}.  Suppose also that $J^t$ is $C^6$ in the deformation parameter $t$.  Define $\mS$ by~\eqref{eqn:S}.  Then $\mS$ is closed.
\end{prop}

\begin{proof}

Let $t_n\in\mS$ with $t_n\to t_0$.  From Corollary~\ref{cor:2.27} we have that
\[ \dim W^{t_0} \leq \liminf_{t_n\to t_0} \dim W^{t_n} . \]
Since $t_n\in\mS$, we have that $\dim W^{t_n}=0$ for all $n$.  Thus $\dim W^{t_0}=0$; i.e.\ $\ker P_4^{t_0}=\mP^{t_0}$.  It follows easily that $t_0\in\mS$.
\end{proof}

Proposition~\ref{prop:2.9} and Proposition~\ref{prop:2.29} together prove Theorem~\ref{thm:main_thm}.
\section{The proof of Corollary~\ref{cor:convex}}
\label{sec:cor}

As discussed in the introduction, a calculation from~\cite{ChanilloChiuYang2013} and stability results from~\cite{ChanilloChiuYang2010,Lempert1994} show that real ellipsoids in $\bC^2$, viewed as deformations of the standard CR three-sphere, satisfy conditions (1), (3), (4) and (5) of Theorem~\ref{thm:main_thm}.  The continuity of the Szeg{\H o} projectors $S^t\colon F^{2,0}\to\ker\bar\partial_b$ for this family follows from results of Kerzman and Stein~\cite{KerzmanStein1978}, as we now explain.

Let $\Omega\subset\bC^2$ be a strictly pseudoconvex domain with boundary $M=\partial\Omega$ and let $H\colon M\times M\to\bC$ be the Henkin--Ramirez kernel.  Define the operator ${\bf H}\colon L^2(M)\to L^2(M)\cap\ker\bar\partial_b$ by
\[ {\bf H}u(w) = \int_M H(w,z)u(z)\, d\sigma_z . \]
Let ${\bf A}={\bf H}^\ast - {\bf H}$.  From the reproducing properties of ${\bf H}$ and the Szeg{\H o} projection $S$, Kerzman and Stein observe that the Szeg{\H o} projection $S$ can be written
\[ S = {\bf H}\left(\mathrm{I}-{\bf A}\right)^{-1} ; \]
see~\cite[(3.4.7)]{KerzmanStein1978}.  Thus the continuity of the Szeg{\H o} projector follows from the continuity of the Henkin--Ramirez kernel.  Finally, \cite[Theorem~1.3.1]{KerzmanStein1978} implies that the Henkin--Ramirez kernel is continuous for families of real ellipsoids in $\bC^2$.

Suppose now that $\omega$ is an $L^2$ section of $F^{2,0}$.  Since $\Omega\subset\bC^2$, we can choose a global frame $Z_1$, and thus consider $f=\iota_{Z_1}i_T\omega\in L^2$.  That is, $\omega=f\theta^1\wedge\theta$.  Then $S(f)\in\ker\bar\partial_b$.  From the structure relations $d\theta=ih_{1\bar 1}\theta^1\wedge\theta^{\bar 1}$ and $d\theta^1=\theta^1\wedge\omega_1{}^1 - A_{\bar 1}{}^1\theta^{\bar1}\wedge\theta$, we observe that
\[ S(\omega) := S(f)\theta^1\wedge\theta \]
is $\bar\partial_b$-closed.  It follows readily that $S\colon F^{2,0}\to\ker\bar\partial_b$ is an orthogonal projection; i.e.\ $S$ is the Szeg{\H o} projector of Theorem~\ref{thm:main_thm}.  It then follows from the previous paragraph that the real ellipsoids also satisfy condition (2) of Theorem~\ref{thm:main_thm}, and hence the conclusion of Corollary~\ref{cor:convex} holds.

\bibliographystyle{abbrv}
\bibliography{../bib}
\end{document}